\documentclass[conference, letterpaper, 10 pt]{IEEEtran}
\IEEEoverridecommandlockouts
\usepackage[letterpaper, margin=.75in]{geometry}

\usepackage[utf8]{inputenc}
\usepackage[T1]{fontenc}
\usepackage[american]{babel}
\usepackage[center]{caption}
\usepackage{graphicx}
\usepackage{listings}
\usepackage{float}
\usepackage{amsmath,amssymb,exscale}
\usepackage{blindtext, graphicx}
\usepackage{verbatim}
\usepackage{algorithm}
\usepackage{algpseudocode}
\usepackage{fancyvrb}
\usepackage{bera}
\usepackage{amsmath}
\usepackage{mathtools}
\usepackage{lipsum}
\usepackage{amsthm,amssymb}
\usepackage{cite}

\usepackage{times}

\usepackage{xcolor}
\usepackage{subcaption}

\usepackage{scalerel}

\usepackage[bookmarks=false]{hyperref}

\def\BibTeX{{\rm B\kern-.05em{\sc i\kern-.025em b}\kern-.08em
    T\kern-.1667em\lower.7ex\hbox{E}\kern-.125emX}}
 
\makeatletter 
\let\old@ps@headings\ps@headings 
\let\old@ps@IEEEtitlepagestyle\ps@IEEEtitlepagestyle 
\def\confheader#1{% 
\def\ps@headings{% 
\old@ps@headings% 
\def\@oddhead{\strut\hfill#1\hfill\strut}% 
\def\@evenhead{\strut\hfill#1\hfill\strut}% 
}% 
\def\ps@IEEEtitlepagestyle{% 
\old@ps@IEEEtitlepagestyle% 
\def\@oddhead{\strut\hfill#1\hfill\strut}% 
\def\@evenhead{\strut\hfill#1\hfill\strut}% 
}% 
\ps@headings% 
} 
\makeatother

\newcommand{\rhob}{\rho}
\newcommand{\flab}[1]{\label{fig:#1}}
\newcommand{\fig}[1]{Fig.\ref{fig:#1}}

\newcommand{\elab}[1]{\label{eqn:#1}}
\newcommand{\eqn}[1]{(\ref{eqn:#1})}

\newcommand{\real}{\mathbb{R}}

\newtheorem{theorem}{Theorem}
\newtheorem*{remark}{Remark}

\newcommand{\diag}[1]{\textbf{diag}\left(#1\right)}
\newcommand{\sym}[1]{\textbf{sym}\big(#1\big)}

\newcommand{\vo}[1]{\boldsymbol{#1}}

\begin{document}

\title{Optimal Sensing Precision for Celestial Navigation Systems in Cislunar Space using LPV Framework}

\author{\IEEEauthorblockN{Eliot Nychka\thanks{Graduate Student, \texttt{eliot.nychka@tamu.edu}} \hspace{0.5in} Raktim Bhattacharya \thanks{Professor, \texttt{raktim@tamu.edu}}\vspace{.1in} }
\IEEEauthorblockA{Aerospace Engineering, Texas A\&M University,\\ College Station, TX, 77843-3141.
}}

\newgeometry{top=1in, bottom=0.75in, left=0.75in, right=0.75in}
\maketitle

\begin{abstract}
This paper introduces two innovative convex optimization formulations to simultaneously optimize the $\mathcal{H}_2/\mathcal{H}_\infty$ observer gain and sensing precision, and guarantee a specified estimation error bound for nonlinear systems in LPV form. Applied to the design of an onboard celestial navigation system for cislunar operations, these formulations demonstrate the ability to maintain accurate spacecraft positioning with minimal measurements and theoretical performance guarantees by design.
\end{abstract}


\section{Introduction}
The recent surge in interest toward cislunar space -- the region within the gravitational spheres of influence of Earth and the Moon -- is fueled by military and commercial motivations. The strategic significance of cislunar space lies in its potential to host assets for advanced surveillance and communication systems that exceed the capabilities within geostationary orbit, thus enhancing national security. From a commercial perspective, cislunar space represents a frontier of untapped potential, promising lucrative opportunities in areas such as resource extraction, space tourism, and as a launchpad for further solar system exploration. These prospects underscore the substantial economic implications of cislunar development.

Space Domain Awareness (SDA) in cislunar space is crucial as activities within this region increase. The strategic and economic significance of cislunar space necessitates sophisticated SDA capabilities to ensure the safety of navigational operations and the integrity of space assets. Effective SDA in cislunar space involves comprehensive monitoring, tracking, and predictive analytics to manage space traffic, debris, and potential threats, which are essential for collision avoidance and operational security.

Recently, planar Earth-Moon resonant orbits have been proposed as a feasible strategy for comprehensive cislunar SDA (CSDA) \cite{frueh2021cislunar}. Such a scenario is shown in \fig{constellation}, where surveillance platforms (shown as black dots) move along a candidate resonant orbit in sufficient large numbers to ensure necessary coverage, for example, to track a resident space object (RSO) shown in magenta.
\begin{figure}[h!]
    \centering
    \includegraphics[width=0.45\textwidth]{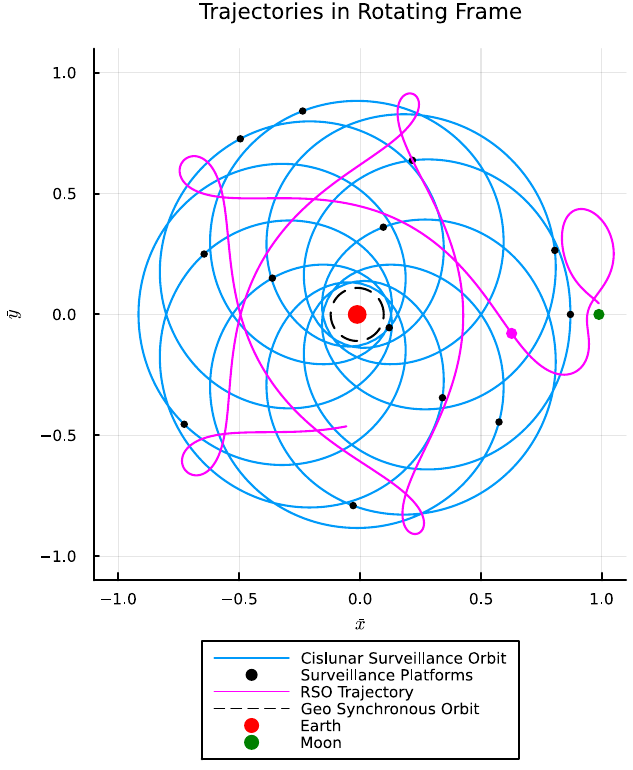}
    \caption{Earth-Moon resonant orbits in cislunar space for tracking resident space objects (RSOs).}
\flab{constellation}
\end{figure}

The reliability of such surveillance systems depends on the ability of each surveillance agent to navigate robustly in the cislunar space, which is challenging due to the complex dynamic interaction between the Earth and the Moon and limited sensing abilities. The gravitational interactions between the Earth, Moon, and spacecraft lead to non-linear and chaotic dynamics, particularly in regions influenced by Earth's and Moon's gravity. Celestial navigation is a proven method in deep space and relies on observations of stars, the Sun, the Earth, and the Moon to determine the spacecraft’s position and orientation. However, its application in cislunar space requires high precision in sensor technology and sophisticated algorithms robust to chaotic cislunar dynamics. 
\restoregeometry
State-of-the-art celestial navigation systems for spacecraft use advanced algorithms and technologies to achieve accurate position and orientation estimates, essential for deep-space missions. These systems predominantly use star trackers and sextants to observe celestial bodies and employ sophisticated filtering techniques like the Extended Kalman Filter (EKF) and the Unscented Kalman Filter (UKF) to handle nonlinearities and uncertainties inherent in the system \cite{ning2007autonomous, ning2008spacecraft}. Additionally, Particle Filters are implemented to manage non-Gaussian noise environments, providing robustness in complex scenarios \cite{ning2008spacecraft}. Since these approaches are based on several approximations, e.g., linearization, Gaussian uncertainty models, and sample approximations, they do not provide global performance guarantees. The practical effectiveness in specific applications is often validated through experimental setups and simulations tailored to the particular characteristics of the system being modeled. Our recent work has developed an LPV (linear parameter varying) framework for developing celestial navigation systems for cislunar applications \cite{eliot-jgcd} with guaranteed bounds on the estimator's error in the $\mathcal{L}_\infty/\mathcal{L}_2$ sense.

The issue with these methods is that the sensor precision needs to be specified as input to the design, and the estimator's accuracy depends on it. 
In celestial navigation systems designed for spacecraft, the configuration and quantity of sensors vary significantly based on mission objectives, spacecraft design, and the operational environment. Commonly incorporated sensors include multiple star trackers for reliable star positioning, which is essential for maintaining orientation. Sun and Earth sensors determine the spacecraft's orientation relative to the Sun and Earth, providing vital data for orbits close to Earth. Inertial Measurement Units (IMUs), which consist of accelerometers and gyroscopes, are necessary for tracking changes in attitude and velocity. Magnetometers are often used in Earth orbit to measure magnetic fields, aiding in attitude control. Additional sensors like optical devices for detailed celestial observations and radio navigation systems for precise positioning using Earth-based or inter-satellite signals may also be included. The key challenge here is to determine the precision of these sensors to meet a given mission's error budget. Often, unnecessarily high-precision sensors are used, resulting in expensive systems. 

This paper addresses the challenge by simultaneously optimizing sensor precision and synthesizing the corresponding estimator that ensures user-specified estimation accuracy. We address this issue by expanding the LPV formulation presented in \cite{eliot-jgcd} to include sensor precision optimization. To the best of our knowledge, this is the first paper that addresses this problem. 

\section{Specific Contributions}
The primary contributions of this paper are two new convex optimization formulations to determine a nonlinear state estimator in LPV form and minimum sensing precision that achieves a user-specified bound on estimation error. The results are presented for $\mathcal{H}_2$ and $\mathcal{H}_\infty$ robust estimation framework and detailed in \S\ref{sec:sparse_h2} and \S\ref{sec:sparse_hinf}.

%\section{Related Works}
%\input{lit_survey}

\section{Problem Formulation}
\subsection{Brief Overview of LPV Systems}
Linear Parameter Varying (LPV)\cite{shamma1990analysis, shamma2012overview} control systems represent an advanced approach in control theory for handling nonlinear systems by employing a family of linear controllers. These controllers are scheduled based on varying parameters that are measurable in real time. The LPV framework extends the gain scheduling approach, which traditionally involves designing controllers at various operating points and interpolating between them as system parameters change. The interpolated controller does not have theoretical guarantees. However, the LPV framework provides theoretical guarantees since the design accounts for parameter variations. Several methodologies exist for designing LPV controllers, including \textit{linear fractional transformations (LFT)}, \textit{single quadratic Lyapunov function (SQLF)}, and \textit{parameter dependent quadratic Lyapunov function (PDQLF)}. These methodologies reformulate control design problems into convex optimization problems involving linear matrix inequalities (LMIs) \cite{el2000advances,helton2007linear}. The main challenge is representing the nonlinear system as a linear parameter-varying system. For example, a nonlinear system $\dot{x} = x^3$ can be written as $\dot{x} =A(\rho)x$, where $A(\rho) = \rho$ with $\rho(t) := x^2(t)$. When LPV models are affine in the parameter, i.e., $A(\rho) = A_0 + A_1\rho$, the constraints to guarantee performance for all variation in $\rho(t)$ are simpler. It is assumed $\rho(t)$ is in a convex polytope, and guarantees are established with constraints defined on the vertices of the polytope. If the dependence on $\rho(t)$ is nonlinear, randomized algorithms are often employed with probabilistic guarantees \cite{tempo2013randomized, fujisaki2003probabilistic}.  Often, additional parameters are introduced to force an affine structure, resulting in many parameters. The computational burden increases with problem size and is dominated by the worst-case complexity of LMIs are $\mathcal{O}(n^6)$, where $n$ is the problem size. Also, restricting $\rho(t)$ in a closed set is often not trivial. Besides these limitations, the LPV framework has been quite successful, particularly in the aerospace industry \cite{ganguli2002reconfigurable, balas2002linear, gilbert2010polynomial, marcos2009lpv}.

\subsection{LPV Observer Design}
% Based on the formulation denoted in \cite{9121301} the formulation for the system state space is as follows. It should be noted that there is no \(u(t)\) vector because this application assumes no control, only observation.

We consider the following LPV system
\begin{subequations}
\begin{align}
 {\dot{x}}(t)&= {A}(\rhob)  {x}(t) +   {b}(\rhob) +  {B_w}(\rhob)  {w}(t),\\
 {y}(t) &= {C}(\rhob) {x}(t) +  {d}(\rhob) +  {D_w}(\rhob)  {w}(t),\\
 {z}(t) &= {C_z x}(t),
\end{align}
\elab{State_Space_LPV}
\end{subequations}
where \( {x}\in \mathbb{R}^{N_x}\), \( {y}\in \mathbb{R}^{N_y}\) are respectively the state vector, the vector of measured outputs, and the output vector of interest, and $\rho\in\mathbb{R}^{N_\rho}$ is the parameter on which the dynamics is scheduled. Also, \( {w} \in \mathbb{R}^{N_w}\) is the vector of the exogenous signals and is noted as:
\begin{equation}
 {w}(t)=
\begin{bmatrix}
 {d}(t) \\
 {n}(t)
\end{bmatrix},
\end{equation}  
where $d\in\real^{N_d}$ is the disturbance and $n\in\real^{N_y}$ is the measurement noise.

\begin{remark}
Note that the dynamics in \eqn{State_Space_LPV} has the terms ${b}(\rhob)$ and ${d}(\rhob)$, which admits nonlinear systems that cannot be expressed as LPV but as affine parameter varying systems. The nonlinear system in the classical cislunar CR3BP (circular restricted three-body problem) is one such system.
\end{remark}

The full-order \textit{nonlinear} state observer in LPV form for the system \eqn{State_Space_LPV} is given by
\begin{subequations}
\begin{align}
 {\dot{\hat{x}}}(t) &= \Big(A(\rhob) +  LC_y(\rhob)\Big)\hat{x}(t) -  Ly(t) +  \Big(b(\rhob) +  Ld(\rhob)\Big),\\
 {\hat{z}} &=  {C_z \hat{x}}(t),
\end{align}
\elab{estimator_dynamics}
\end{subequations}
where \( {\hat{x}} \in \mathbb{R}^{N_z}\) is the estimated state vector, \( {\hat{x}} \in \mathbb{R}^{N_z}\) is the estimate of the vector of interest, and \( {L} \in \mathbb{R}^{N_z \times N_y}\) is the observer gain. The error vector is defined as 
\begin{equation}
 {e}(t)=  {x}(t)- {\hat{x}}(t), \text{and }  {\epsilon}(t)=  {z(t)}- {\hat{z}}(t).
\end{equation}  

\begin{remark}
Note that the estimator dynamics in \eqn{estimator_dynamics} is different from the conventional formulation as it has the ${b}(\rhob)$ and ${d}(\rhob)$ terms to account for the affine form of the LPV system in \eqn{State_Space_LPV}.
\end{remark}
The error dynamics is then given by
\begin{subequations}
\begin{align}
{\dot{e}}(t) &= (A(\rhob)+LC_y(\rhob)) e(t) \notag \\ & + (\begin{bmatrix}B_d(\rhob)S_d & 0\end{bmatrix} + L\begin{bmatrix}D_d(\rhob)S_d & S_n\end{bmatrix})\bar{w}(t), \\
{\epsilon}(t) &= C_z {e}(t),
\end{align} 
\end{subequations}
where $\bar{w}$ is the normalized exogeneous input with $\|\bar{w}\|_2 = 1$. The variables $S_d$ and $S_n$ scale the inputs to the physical space. In particular, $S_n$ scales the sensor noise, which is an optimization variable in the subsequent formulation.

The transfer function from $\bar{w}(t)$ to error $\epsilon(t)$ is then given by,
\begin{equation}
G_{ \bar{w} \rightarrow \epsilon}(s,\rho) := C_z(sI_{Nx}-A(\rhob)-LC_z(\rhob))^{-1}(B_{ \bar{w}}+LD_{ \bar{w}}),
\end{equation}  
where \(B_{ \bar{w}}(\rhob) = [B_d(\rhob) S_d \textbf{ } 0]\) and  \(D_{ \bar{w}}(\rhob) = [B_d(\rhob) S_d \textbf{ } S_n]\).

We are interested in determining the gain $L$, such that $\|G_{ \bar{w} \rightarrow \epsilon}(s,\rho)\|_p<\gamma$ for $p=2,\infty$, while maximizing the noise in the sensor. The sensor noise $n(t):=S_n\bar{n}(t)$, where $\|\bar{n}(t)\|_2=1$ and $S_n := \diag{\begin{matrix}\kappa_1 & \cdots & \kappa_{N_y}\end{matrix}}$, with $\kappa_i:= 1/\|n_i(t)\|_2$ defined as the precision of the $i^\text{th}$ sensor. We also define $\beta_i := \kappa_i^2$, and the corresponding vector $\beta := \begin{bmatrix} \beta_1 & \cdots & \beta_{N_y}\end{bmatrix}^T$. It turns out that the problem is convex in $\beta$ for $p=2,\infty$, as shown in the next two sections.

\subsection{Optimal Sensing Precision for $\mathcal{H}_2$-LPV Observers} \label{sec:sparse_h2}
\begin{theorem}
The optimal sensing precision and observer gain for an LPV estimator with $\|G_{ \bar{w} \rightarrow \epsilon}(s,\rho)\|_2<\gamma$ is given by the convex optimization problem
\begin{equation}
    \left.
        \begin{aligned}
            &\min_{Y,Q>0,X>0,\beta>0} \quad  ||\beta||_p \quad \text{such that},  \\ 
            &M_{11} = XA(\rhob) + YC_y(\rhob) \\ & \quad  \quad +(XA((\rhob)+YC_y((\rhob))^T, \\ 
            &M_{12} = XB_d(\rhob)S_d + YCD_d(\rhob)S_d, \\
            &
            \begin{bmatrix}
                M_{11} & M_{12} & Y \\ 
                M_{12}^T & -I_{N_d} & 0 \\ 
                Y^T & 0 & -\text{diag}(\beta)
            \end{bmatrix} < 0, \\
            & 
            \begin{bmatrix}
                -Q & C_z(\rhob) \\ 
                C_z^T(\rhob) & -X 
            \end{bmatrix} < 0, \\
            &\text{trace}(Q) < \gamma^2
        \end{aligned}
    \right\}  
\label{proof 1 1st lmi}
\end{equation}
where the optimal sensing precision $\kappa_i = \sqrt{\beta_i}$, and $L := X^{-1}Y$.
\end{theorem}

\begin{proof} The condition \(||G_{ \bar{w}\rightarrow\epsilon}(s)||_2<\gamma\) is equivalent to the existence of a symmetric matrix \(P>0\) such that \cite{9121301} 

\begin{subequations}
\begin{equation}
\begin{split}
(A+LC_y(\rhob)P+P(A+LC_y(\rhob))^T \\
+ (B_{ \bar{w}}(\rhob)+LB_{ \bar{w}}(\rhob))(B_{ \bar{w}}(\rhob)+LB_{ \bar{w}}(\rhob))^T < 0 
\label{proof 1 h2}
\end{split}
\end{equation}
\begin{equation}
\text{trace}(C_z(\rhob)PC_z^T(\rhob)) < \gamma
\elab{proof 2 h2}
\end{equation}
\end{subequations}

Pre- and post-multiplying \ref{proof 1 h2} by \(P^{-1}\) gives

\begin{equation}
\begin{split}
P^{-1}(A+LC_y(\rhob)+(A+LC_y(\rhob))^T P^{-1} \\
+ P^{-1}(B_{ \bar{w}}(\rhob)+LB_{ \bar{w}}(\rhob))(B_{ \bar{w}}(\rhob)+LB_{ \bar{w}}(\rhob))^T P^{-1} < 0 
\end{split}
\end{equation}

Letting \(X:=P^{-1}\) and \(Y:=XL\) in the previous equation to get
\begin{equation}
\begin{split}
(XA+YC_y(\rhob)+(XA+YC_y(\rhob))^T \\
+ (XB_{ \bar{w}}(\rhob)+YB_{ \bar{w}}(\rhob))(XB_{ \bar{w}}(\rhob)+YB_{ \bar{w}}(\rhob))^T < 0 
\end{split}
\elab{equality1}
\end{equation}
Then using the definition of \(B_{ \bar{w}}(\rhob)\) and \(D_{ \bar{w}}(\rhob)\), and  \(M_{11}:= (XA+YC_y(\rhob)+(XA+YC_y(\rhob))^T\), and defining \(M_{12} := XB_d(\rhob)S_d + YCD_d(\rhob)S_d\) the
inequality in \eqn{equality1} can be written as 

\begin{subequations}
\begin{equation}
    M_{11} + \begin{bmatrix}M_{12} & YS_n\end{bmatrix} 
    \begin{bmatrix}
        M^T_{12} \\
        S_n^TY^T
    \end{bmatrix} < 0, \text{or}
\end{equation}
\begin{equation}    
    M_{11} + [M_{12} \text{  } YS_n]
    \begin{bmatrix}
        I_{N_d} & 0 \\
        0 & S_n S_n^T
    \end{bmatrix}
    \begin{bmatrix}
        M^T_{12} \\
        S_n^TY^T
    \end{bmatrix} < 0.
\end{equation}
\end{subequations}

Applying Schur's complement and substituting $\beta_i := \kappa_i^2$ the inequality becomes,
\begin{equation}
\begin{bmatrix}
    M_{11} & M_{12} & Y \\ 
    M_{12}^T & -I_{N_d} & 0 \\ 
    Y^T & 0 & -\text{diag}(\beta)
\end{bmatrix} < 0, \\
\elab{lmi1}
\end{equation}
Now applying Schur's complement to the inequality \eqn{proof 2 h2}, we get \(C_z(\rhob)PC_z^T(\rhob)-Q <0\) and \(\text{trace}(Q) < \gamma^2\), for $Q>0$. Finally, using Schur's complement lemma and substituting \(P^{-1} = X\) we get
\begin{equation}
\begin{bmatrix}
    -Q & C_z(\rhob) \\ 
    C_z^T(\rhob) & -X 
\end{bmatrix} < 0, \text{trace}(Q) < \gamma^2.
\elab{lmi2}
\end{equation}

Therefore the convex optimization formulation is given by
\begin{align*}
\min_{Y,Q>0,X>0,\beta>0}, \text{subject to \eqn{lmi1} and \eqn{lmi2}}      
\end{align*}
\end{proof}

\subsection{Optimal Sensing Precision for $\mathcal{H}_\infty$-LPV Observers} \label{sec:sparse_hinf}
\label{sec:sparse_hinf}
\begin{theorem}
The optimal sensing precision and observer gain for an LPV estimator with $\|G_{ \bar{w} \rightarrow \epsilon}(s,\rho)\|_\infty<\gamma$ is given by the convex optimization problem
\begin{equation}
    \left.
        \begin{aligned}
            &\min_{Y,Q>0,X>0,\beta>0} \|\beta\|_p \text{ such that},  \\ 
            &M_{11}(\rhob) = \sym{XA(\rhob) + YC_y(\rhob)},\\
            &M_{12}(\rhob) = XB_d(\rhob)S_d + YCD_d(\rhob)S_d, \\
            &
            \begin{bmatrix}
    M_{11}(\rhob) & M_{12}(\rhob) & C_z^T(\rhob) & Y \\ 
    M_{12}^T(\rhob) & -\gamma^2 I_{N_d} & 0 & 0 \\ 
    C_z(\rhob) & 0 & - I_{N_z} & 0 \\ 
    Y^T & 0 & 0 & -\gamma^2\textbf{diag}(\beta)
\end{bmatrix} < 0,
        \end{aligned}
    \right\}  
\label{proof 2 1st lmi}
\end{equation}
where the optimal precision is given by $\kappa_i = \sqrt{\beta_i/\gamma}$, and \(L = X^{-1}Y\). 
\end{theorem}

\begin{proof}
The condition $\|G_{\bar{w}\rightarrow\epsilon}(s,\rho)\|_\infty<\gamma$ is equivalent to a symmetric matrix \(X > 0\) such that \cite{9121301}
\begin{equation}
\begin{bmatrix}
    \sym{XA(\rhob) + XLC_y(\rhob)} & \\
     + C_z(\rhob)C_z(\rhob)^T & X(B_{ \bar{w}}(\rhob)+LD_{ \bar{w}}(\rhob)) \\
    (B_{ \bar{w}}(\rhob)+LD_{ \bar{w}}(\rhob))^T X & -\gamma^2I_{(N_d + N_y)}.
\end{bmatrix}<0
\label{lmi proof 3}
\end{equation}
Defining \(Y := XL\) and the \ref{lmi proof 3} equation becomes
\begin{equation}
\begin{bmatrix}
    \sym{XA(\rhob) + XLC_y(\rhob)} & \\
     + C_z(\rhob)C_z(\rhob)^T & XB_{ \bar{w}}(\rhob)+YD_{ \bar{w}}(\rhob) \\
    (XB_{ \bar{w}}(\rhob)+YD_{ \bar{w}}(\rhob))^T X & -\gamma^2I_{(N_d + N_y)},
\end{bmatrix}<0
\label{lmi proof 4}
\end{equation}
where $\sym{\cdot} := (\cdot) + (\cdot)^T$. 

Using Schur's compliment, the inequality can be rewritten as
\begin{align}
&\sym{XA(\rho)+YC_y(\rho)} + C_{z}^T(\rho)C_{z}(\rho) \notag \\ &+ \gamma^{-2}(XB_{\bar{w}}(\rho) + YD_{ \bar{w}}(\rho))(XB_{\bar{w}}(\rho) + YD_{ \bar{w}}(\rho))^T
< 0
\label{lmi part 1}
\end{align}
Using the definitions of \(B_{ \bar{w}}(\rho)\) and \(D_{ \bar{w}}(\rho)\) as well as defining \(M_{11}(\rhob):= \textbf{sym}(XA(\rho)+YC_y(\rho))\), and \(M_{12}(\rhob):= XB_d(\rho)S_d + YB_d(\rho)S_d\), the inequality \ref{lmi part 1} becomes

\begin{equation}
M_{11}(\rhob) + \begin{bmatrix}
M_{12}(\rhob) & C_z^T(\rho) & Y
\end{bmatrix} R 
\begin{bmatrix}
M_{12}^T(\rhob) \\
C_Z(\rho)\\
Y^T
\end{bmatrix} < 0.
\end{equation}
Where \(R := \textbf{diag}(\gamma^{-2}I_{N_d},I_{N_z},\gamma^{-2}S_n S_n^T)\). we can then use the fact \(\kappa^2 = \beta = S_n S_n^T\) and then use Schur's compliment to get. 
\begin{equation}
\begin{bmatrix}
    M_{11}(\rhob) & M_{12}(\rhob) & C_z^T(\rhob) & Y \\ 
    M_{12}^T(\rhob) & -\gamma^2 I_{N_d} & 0 & 0 \\ 
    C_z(\rhob) & 0 & - I_{N_z} & 0 \\ 
    Y^T & 0 & 0 & -\gamma^2\textbf{diag}(\beta)
\end{bmatrix} < 0.
\label{proof 2 second to last lmi}
\end{equation}

%\alert{There still exists some non linear terms in the inequality \ref{proof 2 second to last lmi}. To rectify this \(X' := \gamma X\), \(Y' := \gamma Y\), and \(M_{11}' := \textbf{sym}(XA(\rhob) + YC_y(\rhob)+(XA((\rhob) +YC_y((\rhob))^T)\) and \(M_{12}(\rhob) := X'B_d(\rhob)S_d + Y'CD_d(\rhob)S_d\). Utilizing this notation and pre and post multiplying by }

%\alert{
%\begin{equation}
%F := \textbf{diag}(\frac{1}{\sqrt{\gamma}}I_{N_x},\frac{1}{\sqrt{\gamma}}I_{N_d},\sqrt{\gamma}I_{N_x},\frac{1}{\sqrt{\gamma}}I_{N_y})
%\end{equation}
%leads to the following LMI:
%\begin{equation}
%\begin{bmatrix}
%    M_{11}'(\rhob) & M_{12}'(\rhob) & C_z^T(\rhob) & Y' \\ 
%    M_{12}'^T(\rhob) & -\gamma I_{N_d} & 0 & 0 \\ 
%    C_z(\rhob) & 0 & - \gamma I_{N_z} & 0 \\ 
%    Y'^T & 0 & 0 & \textbf{diag}(\beta)
%\end{bmatrix} < 0.
%\elab{proof 2 last lmi}
%\end{equation}}

Therefore the convex optimization formulation in \ref{proof 2 second to last lmi} is the same as the one shown in \ref{proof 2 1st lmi}, i.e., \begin{align*}
\min_{Y,Q>0,X>0,\beta>0}, \text{subject to \ref{proof 2 second to last lmi}}      
\end{align*}

\end{proof}

\begin{remark}
The cost function in the above optimization problems is given by $\|\beta\|_p$. If $p=1$, the optimal solution will be sparse if it exists. This may be useful if a sensor selection problem is solved, since sensors with zero precisions can be eliminated from the design.
\end{remark}

\begin{remark}
The LMIs in the above optimization problems are affine functions of $\rho$. We assume that $\rho$ is in a convex polytope. Enforcing the LMIs at the vertices of the polytope guarantees that the LMIs are satisfied for all $\rho$ in the convex polytope \cite{apkarian1995self}.
\end{remark}

\section{Results}
\subsection{Representation of Cislunar Dynamics in LPV Form}
We consider the classical circular restricted three-body problem (CR3BP). This simplified celestial mechanics model describes the motion of a small body under the gravitational influence of two larger bodies that are orbiting each other in circular orbits. The two larger bodies, called the primary bodies, are usually much more massive than the small body and are assumed to follow fixed circular orbits around their common center of mass due to their gravitational interaction. The small body, which is considered to have negligible mass compared to the other two, does not affect their motion. In the cislunar case, the larger bodies are the Earth and the Moon, and the smaller body is the spacecraft or any resident space object.

The equations of motion for the CR3BP are,
\begin{subequations}
\begin{equation}
\ddot{\Bar{x}} - 2\dot{\Bar{y}} - \Bar{x} = -\frac{1 - \pi_2}{\sigma^3}(\Bar{x} + \pi_2) - \frac{\pi_2}{\psi^3}(\Bar{x} - 1 + \pi_2), 
\end{equation}
\begin{equation}
\ddot{\Bar{y}} + 2\dot{\Bar{x}} - \Bar{y} = -\frac{1 - \pi_2}{\sigma^3}\Bar{y} - \frac{\pi_2}{\psi^3}\Bar{y}, 
\end{equation}
\elab{cr3bp}
\end{subequations}
where \(\pi_1 = \frac{m_1}{m_1 + m_2}, \quad \pi_2 = \frac{m_2}{m_1 + m_2},\) and  
\begin{align*}
    \sigma & := \|\vo{\sigma}\|_2\ , \text{where}  \quad\vo{\sigma} = \frac{\mathbf{r_{13}}}{r_1{}_2} = (\Bar{x} + \pi_2)\hat{i} + \Bar{y}\hat{j}, \\  \psi & := \|\vo{\psi}\|_2\
, \text{where} \quad \vo{\psi} = \frac{\mathbf{r_{23}}}{r_1{}_2} = (\Bar{x} + \pi_2 - 1)\hat{i} + \Bar{y}\hat{j},
\end{align*} \\
and $r_{12}$ is the distance between the Earth and the Moon. The coordinates are normalized with respect to $r_{12}$, that is, \(\Bar{x}\), \(\Bar{y}\) are \(\frac{x}{r_1{}_2}\), \(\frac{y}{r_1{}_2}\), respectively.
We define parameters $\rho := \begin{bmatrix}\rho_1, \cdots,\rho_6 \end{bmatrix}^T$ where \(\rho_1 = \frac{1}{\sigma^3}\), \quad \(\rho_2 = \frac{1}{\psi^3}\), \quad \(\rho_3 = \frac{1}{\sigma}\), \quad \(\rho_4 = \frac{1}{\psi}\).  
The nonlinear equations in \eqn{cr3bp} can now be written as the following LPV system (with affine dependence on the parameters),
\begin{equation}
\begin{split}
\begin{bmatrix}
\dot{\bar{x}} \\
\dot{\bar{y}} \\
\ddot{\bar{x}} \\
\ddot{\bar{y}} 
\end{bmatrix}
&=
\begin{bmatrix}
0 & 0 & 1 & 0  \\
0 & 0 & 0 & 1  \\
\rho_1(\pi_2-1) + 1 & 0 & 0 & 2  \\
0 & \rho_1(\pi_2-1) + 1 & -2 & 0  
\end{bmatrix}
\begin{bmatrix}
\bar{x} \\
\bar{y} \\
\dot{\bar{x}} \\
\dot{\bar{y}} 
\end{bmatrix} \\
& +
\begin{bmatrix}
0 & 0 & 0 & 0  \\
0 & 0 & 0 & 0  \\
 - \rho_2\pi_2 & 0 & 0 & 0  \\
0 & - \rho_2\pi_2 & 0 & 0  
\end{bmatrix}
\begin{bmatrix}
\bar{x} \\
\bar{y} \\
\dot{\bar{x}} \\
\dot{\bar{y}} 
\end{bmatrix} \\
 & + 
\begin{bmatrix}
0 \\
0 \\
\pi_2(\rhob_1-\rhob_2)(\pi_2-1.0) \\
0 
\end{bmatrix}.
\label{state_space_lpv_1}
\end{split}
\end{equation}%
The sensor model is defined as
\begin{equation}
\begin{split}
\begin{bmatrix}
\sin{\theta_1} \\
\cos{\theta_1} \\
\sin{\theta_2} \\
\cos{\theta_2} \\
r_{13}^2 \\
r_{23}^2
\end{bmatrix}
& =
\begin{bmatrix}
0 & \rho_3  \\
\rho_3 & 0  \\
0 & \rho_4  \\
-\rho_4 & 0  \\
\rho_5+2\pi_2 & \rho_6  \\
\rho_5+2\pi_2-2 & \rho_6  
\end{bmatrix}
\begin{bmatrix}
\Bar{x} \\
\Bar{y} 
%\dot{\Bar{z}} \\
\end{bmatrix} 
+
\begin{bmatrix}
0 \\
\rho_3\pi_2 \\
0 \\
-\rho_4(\pi_2-1) \\ 
\pi_2^2 \\ 
(\pi_2-1)^2 
\end{bmatrix},
\end{split}
\label{state_space_lpv_2}
\end{equation}
where $\theta_1$ and $\theta_2$ are angles derived in \cite{eliot-jgcd}. With the sensor model implemented the LPV state space simplifies to:
\begin{equation}
\begin{split}
A(\rhob) &= \begin{bmatrix}
0 & 0 & 1 & 0  \\
0 & 0 & 0 & 1  \\
a_{31} & 0 & 0 & 2  \\
0 & a_{42} & -2 & 0
\end{bmatrix}, \\
b(\rhob) &=
\begin{bmatrix}
0 \\
0 \\
\pi_2(\rhob_1-\rhob_2)(\pi_2-1.0) \\
0 
\end{bmatrix}, \\
C_y(\rhob) &= \begin{bmatrix}
0 & \rhob_3 & 0 & 0 \\
\rhob_3 & 0 & 0 & 0 \\
0 & \rhob_4 & 0 & 0 \\
-\rhob_4 & 0 & 0 & 0 \\
\rhob_5+2\pi_2 & \rhob_6 & 0 & 0 \\
\rhob_5+2\pi_2-2 & \rhob_6 & 0 & 0 
\end{bmatrix}, \\
d(\rhob) &= \begin{bmatrix}
0 \\
\rhob_3 \\
0 \\
-\rhob_4 \\
\rhob_5+2\pi_2 \\
\rhob_5+2\pi_2-2  
\end{bmatrix},
C_z = \begin{bmatrix}
1 & 0 & 0 & 0 \\
0 & 1 & 0 & 0
\end{bmatrix},
\end{split}
\label{eqn:State_Space}
\end{equation}
where $a_{31} := \rhob_1(\pi_2-1)-\rhob_2\pi_2 + 1$, and $a_{42} := \rhob_1(\pi_2-1) - \rhob_2\pi_2 + 1$.

Note: This system is affine in \(\rho\), meaning that we can use the LMI's derived in \eqref{proof 1 1st lmi} and \eqref{proof 2 1st lmi} to solve the set of LMI's based on the bounds of \(\rho\) to find the maximum sensor degradation this state space can handle for a given performance metric.

\subsection{Simulation Results}
Utilizing the formulation in Theorems 1 and 2, we get the sensing precisions shown in \fig{gamma_kval}.

\begin{figure}[h!]
\centering
\includegraphics[width=0.75\linewidth]{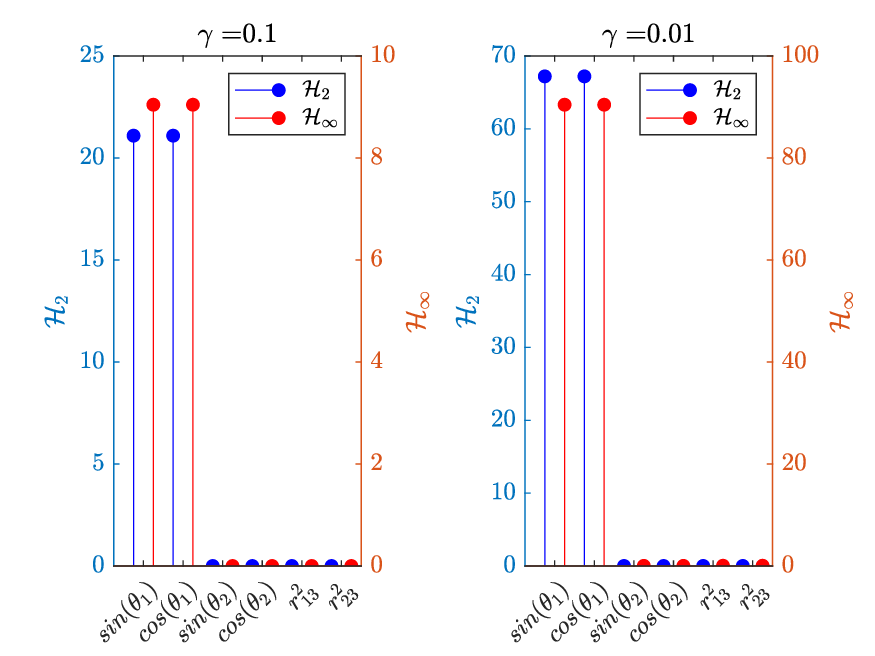}
\caption{Optimal sensing precision $\kappa$.}
\flab{gamma_kval}
\end{figure}

 We observe that the $\mathcal{H}_2$ and $\mathcal{H}_\infty$ observers only require sensors that provide $\sin(\theta_1)$ and $\cos(\theta_1)$. In practice however, $\theta$ is measured from optical systems. Noting that the precision for $\sin(\theta_1)$ and $\cos(\theta_1)$ are identical, we can use either to compute the allowable noise:
\begin{equation}
\begin{split}
\theta_1 & = \sin^{-1}\left(\frac{1}{\kappa_1}\right), \text{ or } \\
\theta_1 & = \frac{\pi}{2} - \cos^{-1}\left(\frac{1}{\kappa_2}\right).
\end{split}
\label{noise_bounds}
\end{equation}
For a \(\gamma = 0.0013\), the maximum sensor noise level that the system can take was found to be around \(0.1^\circ\) for \(\mathcal{H}_2\) and \(0.08^\circ\) for \(\mathcal{H}_\infty\). 

While the optimization indicates that the observer only needs  $\theta_1$ to estimate the states, we need to compute $\rho$ to implement the LPV observer dynamics in practice. It is not possible to determine $\rho$ with only $\theta_1$ information. Therefore, results in \fig{fig:main} utilize both $\theta_1$ and $\theta_2$ to fully calculate $\rho$. The \(0.1^\circ\) of noise calculated in (\ref{noise_bounds}) was applied to both $\theta_1$ and $\theta_2$ for both $\mathcal{H}_2$ and $\mathcal{H}_\infty$. 

Now because $\mathcal{H}_\infty$ simulation uses a larger noise that what was determined, the total error is expected to be larger for $\mathcal{H}_\infty$ which is seen in \fig{h2hinf error plots}. Additionally, due to the singularities introduced by the inverse of \(\sin(\theta_3)\) to find $\rho$ values, the system relies on the observer state estimates to calculate the $\rho$ values in intervals where \(\sin(\theta_3)\) approaches zero. Here $\theta_3 := \pi - \theta_1 - \theta_2$. See  \cite{eliot-jgcd} for details. For this simulation, we chose this region as \(|\sin(\theta_3)|\) < 0.04.

\begin{figure}[htbp]
    \centering
    \begin{subfigure}[b]{0.4\textwidth}
    \centering
    \includegraphics[width=0.8\textwidth]{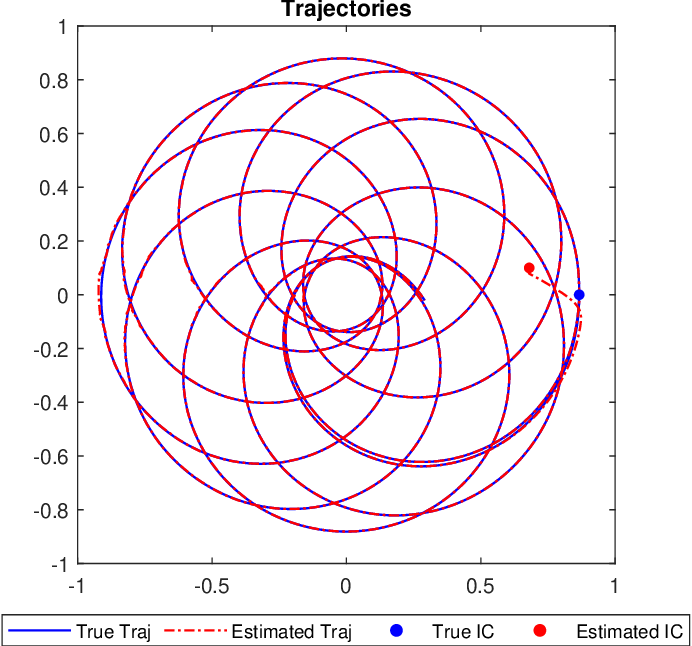}
    \caption{Cislunar trajectory with a $\mathcal{H}_2$ observer with \(0.1^\circ\) noise.}
    \flab{h2 trajectory plots for 0.1 degrees}
    \end{subfigure}
    \hfill
    \begin{subfigure}[b]{0.4\textwidth}
    \centering
    \includegraphics[width=0.8\textwidth]{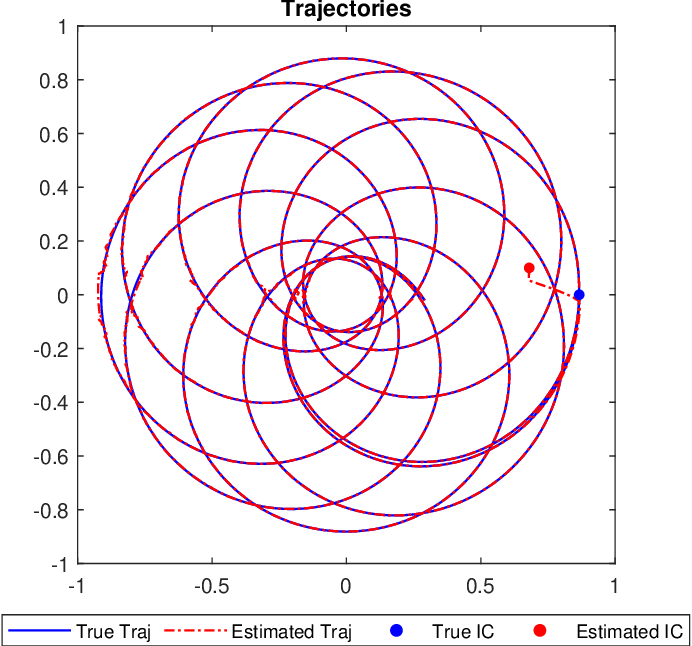}
    \caption{Cislunar trajectory with a $\mathcal{H}_\infty$ observer with \(0.1^\circ\) noise.}
    \flab{hinf trajectory plots for 0.1 degrees}
    \end{subfigure}
    \begin{subfigure}[b]{0.45\textwidth}
    \centering
    \includegraphics[width=\textwidth]{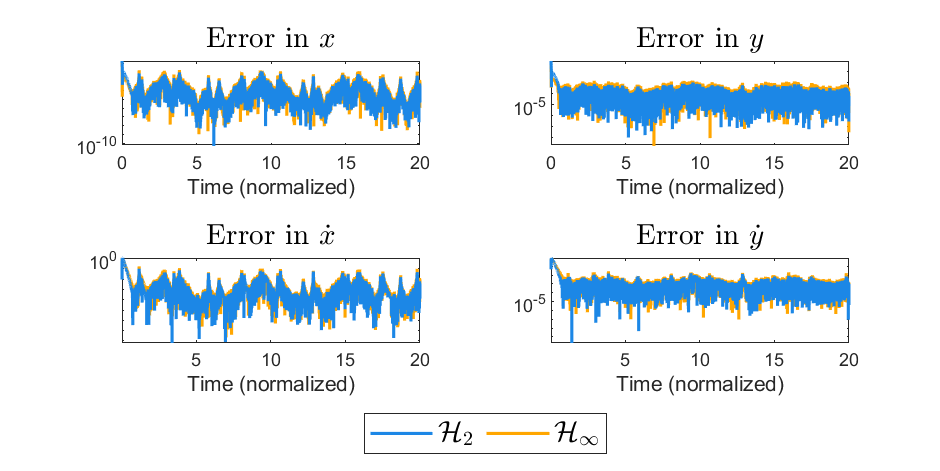}
    \caption{Error comparison between $\mathcal{H}_2$ and $\mathcal{H}_\infty$.}
    \flab{h2hinf error plots}
    \end{subfigure}
    \caption{Results for observers in cislunar space.}
    \flab{fig:main}
\end{figure}
The results in \fig{fig:main} confirm that the $\mathcal{H}_2$ and $\mathcal{H}_\infty$ observers can function without range sensors. Additionally, despite large initial state estimation errors, the observers rapidly converge and maintain accurate estimates with 360 arcseconds ($0.1^\circ$) of noise. 
\vspace{-0.1in}
\section{Conclusions}
In this paper, we presented two innovative convex optimization formulations that simultaneously optimize the $\mathcal{H}_2/\mathcal{H}_\infty$ observer gain and the precision of sensing, guaranteeing a user-specified bound on the estimation error for nonlinear systems modeled as LPV systems. These formulations were applied to the design of an onboard celestial navigation system for cislunar operations. The performance of these methods was confirmed by their ability to accurately predict the position of a spacecraft with minimal sensing. However, the current formulations do not account for sensing occlusions, which may require the integration of additional sensors. The precision of sensing defined by our formulations can guide the design and production of sensing hardware, potentially leading to optimized hardware configurations. In summary, the algorithms formulated in this study offer a robust framework for designing nonlinear state observers in LPV form. These observers are equipped with minimal sensing requirements and are supported by theoretical performance guarantees.

\section{Acknowledgment}
This work is supported by AFOSR grant FA9550-22-1-0539 with Dr. Erik Blasch as the program director, and Sandra and William R. Wheeler '70 Graduate Fellowship in Aerospace Engineering.
\bibliographystyle{IEEEtran}
\bibliography{root.bib}

\end{document}